\documentclass[11pt, leqno]{amsart}

\usepackage[
hmarginratio={1:1},     % equal left and right margins
vmarginratio={1:1},     % equal top and bottom margins
heightrounded,          % always useful
]{geometry}

\usepackage{amssymb}
\usepackage{hyperref}
\def\bbline#1{\mkern1mu\underline{#1\mkern-2mu}\mkern1mu}
\def\btline#1{\mkern2mu\overline{\mkern-2mu#1\mkern-0.5mu}\mkern0.5mu}
\def\bbeta{\bbline{\beta}}

\def\bs#1{\boldsymbol{#1}}
\def\colon{:}
\def\mx{\tilde{\bs x}}
\def\n{\underline{\mathfrak N}} 
\def\N{\overline{\mathfrak N}}
\def\nn{\widetilde{\mathfrak N}}
\def\e{\mathcal E}
\def\epsilon{\varepsilon}
\def\h{\mathcal H}
\def\el#1{E_s(#1)}
\def\fl#1{\{#1\}}
\def\fun{g_{s,d}}
\def\dist{{\rm dist\,}}
\def\diam{{\rm diam}}
\def\g{\bbline g_{s,d}}
\def\G{\btline g_{s,d}}

\def\geq{\geqslant}
\def\leq{\leqslant}
\def\weakto{\stackrel{*}{\longrightarrow}} 
\newtheorem{theorem}{Theorem}[section]
\newtheorem{lemma}[theorem]{Lemma}

\newtheorem{proposition}[theorem]{Proposition}
\newtheorem{corollary}[theorem]{Corollary}
\newtheorem{definition}[theorem]{Definition}

\newtheorem{thm}{Theorem}
\author{A. Reznikov}
\thanks{The research of A.R.\ was supported, in part, by the National Science Foundation grant DMS-1764398}
\author{O. Vlasiuk}
\thanks{O.\,V.\ was supported, in part, by the National Science Foundation grant DMS-1516400.}
\address{Department of Mathematics, Florida State University, Tallahassee, FL 32306}
\email{reznikov@math.fsu.edu}
\address{Center for Constructive Approximation, Department of Mathematics, Vanderbilt University, Nashville, TN 37240\newline
and\newline
Department of Mathematics, Florida State University, Tallahassee, FL 32306}
\email{vlasiuk@math.fsu.edu}

\begin{document}
\title{Riesz energy on self-similar sets} 
\begin{abstract}
    %TODO: adjust abstract
We investigate properties of minimal $N$-point Riesz $s$-energy on fractal sets of non-integer dimension, as well as asymptotic behavior of $N$-point configurations that minimize this energy. For $s$ bigger than the dimension of the set $A$, we constructively prove a negative result concerning the asymptotic behavior (namely, its nonexistence) of the minimal $N$-point Riesz $s$-energy of $A$, but we show that the asymptotic exists over reasonable sub-sequences of $N$. Furthermore, we give a short proof of a result concerning asymptotic behavior of configurations that minimize the discrete Riesz $s$-energy. 
\end{abstract}
%\subjclass[2010]{ Primary, 31C20, 28A78. Secondary, 52A40.}
%\keywords{Riesz energy, equilibrium configurations, external field, covering radius, separation distance, quasi-uniformity.}
\date{\today}
\maketitle 

\vspace{4mm}

\vspace{2mm}


\vspace{2mm}
\section{Introduction}
The minimal energy problem originates from potential theory, where for a compact set $A\subset \mathbb{R}^p$ and a lower semicontinuous kernel $K$ defined on $A\times A$, it is required to find
\begin{equation}\label{eq:contenergy}
I_K(A):=\inf_\mu \int K(x,y) \textup{d}\mu(x) \textup{d}\mu(y),
\end{equation}
where the infimum is taken over all probability measures supported on $A$; moreover, we are interested in the measure that attains this infimum. In this paper we focus on the {\it Riesz $s$-kernels} $K_s(x,y):=|x-y|^{-s}$. It is convenient to discretize the measure on which the value $I_K(A)$ is achieved; for this purpose, we consider the {\it discrete Riesz $s$-energy problem}. Namely, for every integer $N\geq 2$ we define
\begin{equation}\label{eq:discrenergy}
\e_s(A, N):=\inf_{\omega_N} E_s(\omega_N),
\end{equation}
where the infimum is taken over all $N$-point sets $\omega_N=\{\bs x_1, \ldots, \bs x_N\} \subset A$, and
\[
    E_s(\omega_N) := \sum_{i\neq j} |\bs x_i -\bs x_j|^{-s}, \qquad N = 2,3,4,\ldots
\]
Since the kernel $K_s$ is lower semicontinuous, the infimum is always attained. 

In general, asymptotics of energy functionals arising from pairwise interaction in discrete subsets has been the subject of a number of studies \cite{Hardin2005b,Hardin2004, DAMELIN2005845, Brauchart2006}; it has also been considered for random point configurations \cite{Brauchart2015a} and in the context of random processes \cite{Alishahi2015b, Beltran2016b}. The interest in such functionals is primarily motivated by applications in physics and modeling of particle interactions, as well as by the connections to geometric measure theory.

If $d$ is the Hausdorff dimension of $A$ and $s<d$, then there is a unique measure $\mu_{s, A}$ for which the infimum in \eqref{eq:contenergy} is achieved, and the configurations that attain the infimum in \eqref{eq:discrenergy} resemble $\mu_{s, A}$ in the weak$^*$ sense (for the precise definition, see below). When $s>d$, we have $I_{K_s}(A)=\infty$, as the integral in the RHS is infinite on all measures $\mu$ supported on $A$. However, for ``good'' sets $A$ (for example, $d$-rectifiable sets) with integer dimension $ d $, the configurations attaining \eqref{eq:discrenergy} resemble a certain special measure, namely, the uniform measure on $A$. 

More precisely, for a configuration $ \omega_N = \{\bs x_i : 1\leq i \leq N \}\subset A $ we define the (empirical) probability measure 
\[
    \nu_N := \frac{1}{N} \sum_{i=1}^N \delta_{\bs x_i},
\]
and we shall identify the two. Then, as summarized in the \textit{Poppy-seed bagel theorem} (PSB), see Theorem \ref{thm:psb}, under some regularity requirements on the set $A$, any sequence $\{\tilde \omega_N\colon \#\tilde\omega_N = N, \e_s(A, N)=E_s(\tilde \omega_N)\}$  converges to the normalized $ d $-dimensional Hausdorff measure $ \h_d(A\cap\cdot)/\h_d(A) $ on $ A $. Moreover, for such sets $ A $, the following limit exists:
\begin{equation}\label{eq:limittt} \lim_{N\to \infty}\frac{\e_s(A, N)}{N^{1+s/d}}.
\end{equation}

% The requirement of $ A $ being representable as \eqref{eq:assumption} is the weakest that is known to guarantee the existence of \eqref{eq:limittt}; 

On the other hand, it has been established \cite[Proposition 2.6]{Borodachov2007} that for a class of self-similar fractals $ A $ with $ \dim_H A =d  $, the limit of $ \e_s(A,N)/N^{1+s/d} $ does not exist for $ s $ large enough. Using this observation, \cite{Calef2012} gives an example of a set $A$ and a sequence of optimal configurations for $\e_s(A, N)$ without a weak$ ^* $ limit.

In view of the above, it is natural to ask what can be said about weak$ ^* $ cluster points of $ \{\nu_N\colon N\geq 2\} $ in the case when the underlying set $ A $ is not $ d $-rectifiable; a characterization of the cluster points of $\{ \e_s(A, N)/N^{1+s/d}\colon N\geq 2 \} $ is likewise of interest.
%TODO: do we show anything else?  

The \hyperref[sec:prereqs]{following} section contains formal definitions and the necessary prerequisites; Section~\ref{sec:prior} gives an overview of previously established results, both in the case of a rectifiable and a non-rectifiable set $ A $. Sections~\ref{sec:main} and \ref{sec:proofs} contain the formulations of the main theorems and their proofs, respectively.

\section{Self-similarity and open set condition}
\label{sec:prereqs}

We shall be working with subsets of the Euclidean space $ \mathbb R^p $, using bold typeface for its elements: $ \bs x\in \mathbb R^p $. An open ball of radius $ r $, centered at $ \bs x $, will be denoted by $ B(\bs x,r) $. The $d$-dimensional Hausdorff measure of a Borel set $A$ will be denoted by $\h_d(A)$. 

A pair of sets $ A^{(1)},\, A^{(2)} $ will be called \textit{metrically separated} if $ |\bs x - \bs y|\geq \sigma >0  $ whenever $ \bs x\in A^{(1)} $ and $ \bs y \in A^{(2)} $. Recall that a \textit{similitude} $ \psi: \mathbb R^p \to \mathbb R^p $ can be written as
\[
    \psi(\bs x) = r O(\bs x) + \bs z
\]
for an orthogonal matrix $ O \in \mathcal O(p) $, a vector $ \bs z \in\mathbb R^p $, and a contraction ratio $ 0 <r< 1 $.  
The following definition can be found in \cite{Hutchinson}.
\begin{definition}
A compact set $ A\subset \mathbb R^p $ is called a \textit{self-similar fractal} with similitudes $ \left\{\psi_m\right\}_{m=1}^M $ with contraction ratios $ r_m,\, 1\leq m\leq M $ if
%Namely, let $ A \subset \mathbb R^p  $ be the compact set of fixed points of a collection of similitudes $ \left\{\psi_m\right\}_{m=1}^M $ with contraction ratios $ r_m,\, 1\leq m\leq M $, that is, satisfying
\[
    A = \bigcup_{m=1}^M \psi_m(A),
\]
where the union is disjoint\footnote{One also considers self-similar fractals where the union is not disjoint --- these are harder to deal with}.

%?% In particular, such sets can be \textit{purely m-unrectifiable} \cite{MR1333890} for any integer $ m $.

We say that $A$ satisfies the \textit{open set condition} if there exists a bounded open set $ V\subset \mathbb R^p $ such that 
\[
    \bigcup_{m=1}^M \psi_m(V) \subset V,
\]
where the sets in the union are disjoint. 
\end{definition}
For a self-similar fractal $A$, it is known \cite{MR867284, MR1333890} that its Hausdorff dimension $\dim_H A = d$ where $d$ is such that 
\begin{equation}
    \label{eq:dim}
    \sum_{m=1}^M r_m^d =1.  
\end{equation}
It will further be used that if $ A $ is a self-similar fractal satisfying the open set condition, then there holds $ 0 < \h_d(A) < \infty $ and $ A $ is \textit{d-regular} with respect to $ \h_d $; that is, there exists a positive constant $ c $, such that for every $ r,\, 0 < r \leq \diam (A), $ and every $ \bs x \in A $,
\begin{equation}
    \label{eq:dregular}
    c^{-1} r^d \leq \h_d(A\cap B(\bs x, r)) \leq c r^d.
\end{equation}
% From now on we shall only use $ d $-regularity with respect to $ \h_d $, thus always referring to \eqref{eq:dregular}.  

\section{Overview of prior results}
\label{sec:prior}

Recall the standard definition of the weak$ ^* $ convergence: given a countable sequence $ \{\mu_N\colon  N\geq 1\} $ of probability measures supported on $ A $ and another probability measure $ \mu $,
\[
    \mu_N \weakto \mu,\, N\to \infty \quad \Longleftrightarrow \quad \int_A f(\bs x) d\mu_N(\bs x) \longrightarrow \int_A f(\bs x) d\mu(\bs x),\ N\to \infty,
\]
for every $ f\in C(A) $.
(Limits along nets are not necessary, as in this context weak$ ^* $ topology is metrizable.) We shall say that a sequence of discrete sets {\it converges} to a certain measure if the corresponding sequence of counting measures converges to it.

The set $ A $ is said to be {\it $ d $-rectifiable} if it is the image of a compact subset of $ \mathbb R^d $  under a Lipschitz map. Furthermore, we say that $ A $ is {\it $ (\mathcal H_d, d) $-rectifiable}, if
\begin{equation}\label{eq:assumption}
    A = A^{(0)} \cup \bigcup_{k=0}^\infty A^{(k)},
\end{equation}
where for $k\geq 1$ each $ A^{(k)} $ is $ d $-rectifiable and $ \mathcal H_d(A^{(0)}) = 0 $. 

We begin by discussing results dealing with the Riesz energy, both in the rectifiable and non-rectifiable contexts.
To formulate the PSB theorem, suppose $ s>d $ for simplicity; the case of $ s=d $ is similar, but requires stronger assumptions on the set $ A $. We write $ \mathcal M_d(A) $ for the $ d $-dimensional Minkowski content of the set $ A $ \cite[3.2.37--39]{federerGeometric1996}.
\begin{thm}[Poppy-seed bagel theorem, \cite{Hardin2004, Borodachov2008a}]\label{thm:psb}
    If the set $ A $ is $ (\mathcal H_d, d) $-rectifiable for $ s>d $ and $ \mathcal H_d(A) = \mathcal M_d(A) $,  then 
    \[
        \lim_{N\to \infty}\frac{\e(A, N)}{N^{1+s/d}} = \frac{C_{s,d}}{\mathcal H_d(A)^{s/d}},
    \]
    and every sequence $ \{ \tilde\omega_N \colon N\geq 2 \} $ achieving the above limit converges weak$ ^* $ to the uniform probability measure on $ A $:
    \[
        \frac1N \sum\limits_{\mx\in \tilde\omega_N} \delta_{\mx} \weakto \frac{\mathcal H_d(A\cap \cdot)}{\mathcal H_d(A)}.
     \] 
\end{thm}

The smoothness assumptions on $ A $ in the above theorem are essential for existence of the limit of $\mathcal E(A, N)/N^{1+s/d}  $.
Let $ \{ \bbline\omega_N\subset A \colon \#\bbline\omega_N=N, N\in \n\}  $ be a sequence of configurations such that 
\begin{equation}
    \label{eq:lower}
    \lim_{\n\ni N\to\infty} \frac{E_s(\bbline\omega_N)}{N^{1+s/d}} = \liminf_{N\to\infty} \frac{\e_s(A, N)}{N^{1+s/d}} =: \g(A),
\end{equation}
and similarly, $ \{\btline\omega_N\subset A\colon \#\bbline\omega_N=N, N\in \N\}  $ a sequence for which
\begin{equation}
    \label{eq:upper}
    \lim_{\N\ni N\to\infty} \frac{E_s(\btline\omega_N)}{N^{1+s/d}} = \limsup_{N\to\infty} \frac{\e_s(A, N)}{N^{1+s/d}} =: \G(A).
\end{equation}
In the notation of \eqref{eq:lower}-\eqref{eq:upper}, the result about the non-existence of $ \lim_{N\to\infty} \e_s(A,N)/N^{1+s/d}$   from \cite{Borodachov2007} that was mentioned in the introduction can be stated as follows.
\begin{proposition}
    \label{prop:lowupfract}
    For a self-similar fractal $A$ with contraction ratios $r_1=\ldots=r_m$, there exists an $ S_0 > 0 $ such that for every $ s > S_0 $, 
    \[ 
        0 < \g(A) < \G(A) < \infty.
    \]
\end{proposition}
\noindent We remark that in the proof of Proposition~\ref{prop:lowupfract}, the number $S_0$ was not obtained constructively. In Theorem~\ref{thm:prec_lowupfract} we give a formula for $S_0$.  
The behavior of the sets $\omega_N$ that attain $\e_s(A, N)$ in the non-rectifiable case is still not fully characterized. The following proposition, taken from \cite{Calef2012}, is the only known negative result so far. 
%TODO: try two fractals with different $ M $'s
%It is further useful to recall that \cite[Theorem 5.7]{MR1333890} if a compact set is $ d $-regular, it must have Hausdorff dimension $ d $. 
\begin{proposition}
    \label{prop:calef}
    Assume that the two $ d $-regular compact sets $ A^{(1)},\, A^{(2)} $ are metrically separated and are such that $ A^{(1)} $ is a self-similar fractal with equal contraction ratios and $ \g(A^{(2)}) = \G(A^{(2)})  $. Then for any sequence of minimizers $\{\tilde \omega_N\subset A \colon \#\tilde\omega_N=N, E_s(\tilde\omega_N)=\e_s(A, N)\}$, the corresponding sequence of measures 
\[
 \tilde\nu_N = \frac1N \sum\limits_{\mx\in \tilde\omega_N} \delta_{\mx}
 \]
  does not have a weak$ ^* $ limit.
\end{proposition} 
In view of these two propositions, it is remarkable that the local properties of minimizers of $ E_s $ are fully preserved on self-similar fractals.
Indeed, $ d $-regularity of $ A $ can be readily used to obtain that any sequence of minimizers of $ E_s $ has the optimal orders of separation and covering. The following result was proved in \cite{Hardin2012}:
\begin{proposition}
    If $ A\subset \mathbb R^p $ is a compact $ d $-regular set, $ \{ \tilde \omega_N \colon  N\geq 1 \}  $  a sequence of configurations minimizing $ E_s $ with $ \tilde\omega_N = \{\tilde{\bs  x}_i : 1\leq i \leq N \} $, then there exist a constant $ C_1 > 0 $ such that for any $ 1\leq i < j \leq N $,
    \[
        |\mx_i - \mx_j | \geq C_1 N^{-1/d},  \qquad N\geq 2,
    \]
    and a constant $ C_2 >0 $ such that for any $ \bs y \in A $,
    \[ 
        \min_i |\bs y - \mx_i| \leq C_2 N^{-1/d}, \qquad N \geq 2.
    \] 
\end{proposition}
The closest one comes to an analog of the PSB theorem for self-similar fractals is the following proposition \cite{borodachovAsymptotics2012}. Note that we give a simpler proof of \eqref{it:liminfseq} for the case when $ A_0 = A $ in Theorem~\ref{thm:weakstar}.
\begin{proposition}
    Suppose $ A_0 $ is a self-similar fractal satisfying the open set condition and $ s>d $; fix a compact $ A \subset A_0 $.
    \begin{enumerate}
        \item\label{it:liminfseq}
            If $ \{\bbline\omega_N\colon  N\in\n\} $, is a sequence of configurations for which 
            \[
                \lim_{\n\ni N\to\infty} \frac{E_s(\bbline\omega_N)}{N^{1+s/d}} = \g(A),
            \]
            then the corresponding sequence of empirical measures converges weak$ ^* $:
            \[
            \bbline\nu_N \weakto \frac{\h_d(\cdot \cap A)}{\h_d(A)}, \qquad \n\ni N\to \infty.  \]
        \item 
            There holds
            \[
                \g(A) = \frac{\g(A_0)\mathcal H_d(A_0)^{s/d}}{\mathcal H_d(A)^{s/d}}
            \]
            and 
            \[
                \G(A) = \frac{\G(A_0)\mathcal H_d(A_0)^{s/d}}{\mathcal H_d(A)^{s/d}}.
            \] 
    \end{enumerate}
\end{proposition}

We finish this section with another relevant result on fractal sets. In \cite{Borodachov2007} it was shown that, as $s\to \infty$, there is a strong connection between the $s$-energy $\e_s(A)$ and the {\it best-packing constant} 
$$
\delta(A, N):=\sup_{\omega_N} \min_{i\not=j} |\bs x_i - \bs x_j|.
$$
The main theorem of \cite{lalleyPacking1988} is given in terms of the function $N(\delta):=\max\{n\colon \delta(A, n)\geqslant \delta\}$. Our Theorem \ref{thm:well_posed} gives an analog of the second part of this theorem for the minimal discrete energy.
\begin{thm}
    Suppose $ A $ is a self-similar fractal of dimension $d$ satisfying the open set condition with contraction ratios $r_1,\ldots,r_m$.
    \begin{enumerate}
        \item If the additive group generated by $ \log r_1, \ldots, \log r_m $ is dense in $ \mathbb R $, then there exists a constant $ C $ such that
            \[
                \lim_{N\to \infty} N^{1/d} \delta(A, N)=\lim_{\delta\to 0} N(\delta)^{1/d} \delta=C.
            \]
        \item If the additive group generated by $ \log r_1, \ldots, \log r_M $ coincides with the lattice $ h\mathbb Z $ for some $ h>0 $, then
            \[
                \lim N(\delta)^{1/d}\delta=C_\theta, 
            \]
            where the limit is taken over a subsequence $\delta\to 0$ with $\left\{\frac1h \log \delta\right\} = \theta$.
    \end{enumerate}
\end{thm}

%TODO: use or remove
%As already noted, self-similar fractals have the $ d $-regularity property; the same applies to finite unions and countable metrically separated unions of such sets.  

\section{Main results}
\label{sec:main}
%As will become clear from Lemma \ref{lem:locality}, we require that the first order asymptotics of $ \e_s(A, N) $ grow faster than $ N^2 $ when $ N\to\infty $; this is also the case %when $ s = d $, and the results of this section are fully applicable. We shall assume $ s>d $ for simplicity; the case of $ s=d $ is obtained by replacing all the instances of $ N^{1+s/%d} $ with $ N^2\log N $.

In accordance with the prior notation, we write $ \bbline\omega_N = \{\bbline{\bs x}_i : 1\leq i \leq N \} $ for the sequence of configurations with the lowest asymptotics (i.e., such that \eqref{eq:lower} holds), and
\[
    \bbline \nu_N = \frac{1}{N} \sum_{i=1}^N \delta_{\bbline{\bs x}_i}, \quad N\in \n.
\]
As described above, generally the limit of $ \e_s(A, N)/N^{1+s/d}$,  $ N\to\infty $ does not necessarily exist. It is still possible to characterize the behavior of the sequence $ \{\bbline\omega_N\colon  N\in \n\} $. The following result first appeared in \cite{borodachovAsymptotics2012}; we give an independent and a more direct proof.
\begin{theorem}
    \label{thm:weakstar}
    Let $ A\subset \mathbb R^p  $  be a compact self-similar fractal satisfying the open set condition, and $ \dim_H A  = d < s $. If $ \{\bbline\omega_N\colon  N\in\n\} $, is a sequence of configurations for which 
    \[
        \lim_{\n\ni N\to\infty} \frac{E_s(\bbline\omega_N)}{N^{1+s/d}} = \g(A),
    \]
    then the corresponding sequence of empirical measures converges weak$ ^* $:
    \begin{equation}\label{eq:defhd}
        \bbline\nu_N \weakto h_d(\cdot):=\frac{\h_d(\cdot \cap A)}{\h_d(A)}, \qquad \n\ni N\to \infty.
    \end{equation}
\end{theorem} 

When the similitudes $ \{\psi_m \}_{m=1}^M $ fixing $ A $ all have the same contraction ratio, it is natural to expect some additional symmetry of minimizers, associated with the $ M $-fold scale symmetry of $ A $. Similarly, since the energy of interactions between particles in different $ A^{(m)} $ is at most of order $ N^2 $, see proof of Lemma~\ref{lem:locality} below, we expect that by acting with $ \{\psi_m \}_{m=1}^M $ on a minimizer $ \tilde\omega_N $ with $ N $ large, we obtain a near-minimizer with $ MN $ elements. This heuristic is made rigorous in the following theorem.
\begin{theorem}
    \label{thm:subseq}
Let $ A\subset \mathbb R^p $ be a self-similar fractal, fixed under $ M $ similitudes with the same contraction ratio, and $ \mathfrak M = \{M^k n\colon  k\geq 1 \} $. Then the following limit exists
\[
    \lim_{\mathfrak M \ni N\to \infty} \frac{\e_s(A,N)}{N^{1+s/d}}.
\]
\end{theorem}

The previous theorem can be further extended. We shall need some notation first.
% to show that in the case of equal contraction ratios $ r_1 = \ldots = r_m $, any sequence $ \mathfrak N $ for which the corresponding sequence of minimizers has a limit must be of %this form. Indeed, the following result establishes a bijection between cluster points of $ \{\{\log_M N\} \colon  N\in \mathfrak N \} $ and those of the sequence $ \{\mathcal E_s (A,N)/N^{1+s/d} \colon  N\in \mathfrak N \} $. We shall need some notation first.
For a sequence $ \mathfrak N $, let 
\[
    \fl{\mathfrak N} := \lim_{\mathfrak N \ni N \to \infty} \{ \log_M N \},
\]
where $ \{\cdot\} $ in the RHS denotes the fractional part, and
\[
    \el{\mathfrak N} := \lim_{\mathfrak N \ni N \to \infty} \frac{\e_s(A,N)}{N^{1+s/d}},
\]
if the corresponding limit exists.
\begin{theorem}
    \label{thm:well_posed}
    If $ A $ is a  self-similar fractal with equal contraction ratios, and two sequences $ \mathfrak N_1, \, \mathfrak N_2 \subset \mathbb N $ are such that
    \begin{equation}
        \label{eq:logsequences}
        \fl{\mathfrak N_1} = \fl{\mathfrak N_2}, 
    \end{equation}
    then 
    \begin{equation}
        \label{eq:energylogsequences}
        \el {\mathfrak N_1} = \el {\mathfrak N_2}.  
    \end{equation}
    In particular, the limits in \eqref{eq:energylogsequences} exist. Moreover, the function $ \fun: \fl{\mathfrak N}\mapsto \el{\mathfrak N} $ is continuous on $ [0,1] $.
\end{theorem}
In the case of equal contraction ratios, the argument in the proof of Theorem~\ref{thm:subseq}, can  be further used to make the result of Proposition~\ref{prop:lowupfract} more precise.
\begin{theorem}
    \label{thm:prec_lowupfract}
    Let $ A\subset \mathbb R^p $ be a self-similar fractal, fixed under $ M $ similitudes with the same contraction ratio $r$, and write $ \sigma := \min \{ \|\bs x - \bs y\| \colon  \bs x\in A_i,\, \bs y \in A_j, i\neq j \} $. If 
    \[
        R :=\frac r\sigma(1+r^d)^{1/d} < 1,
    \]
    then for for every value of $ s $ such that
    \begin{equation}
        \label{eq:sbound}
        s \geq \max \left\{ 2d,\, \log_{1/R} [2M(M+1)] \right\},
    \end{equation}
    there holds 
    \[
        0 < \g(A) < \G(A) < \infty.
    \]
\end{theorem}
The proof of this theorem requires an estimate for the value of $ \mathcal E_s(A, M) $, which results in the condition $ R < 1 $. When $ \mathcal E_s(A, M) $ can be computed explicitly, a similar conclusion can also be obtained for sets that do not necessarily satisfy $ R<1 $, as in the following.
\begin{corollary}
    \label{cor:ternary_cantor}
    If $ A $ is the ternary Cantor set and $ s > 3 \dim_H A = 3 \log_3 2 $, then 
    \[
        0 < \g(A) < \G(A) < \infty.
    \]
\end{corollary}

\section{Proofs}\label{sec:proofs}
The key to proving Theorem \ref{thm:weakstar} is that the hypersingular Riesz energy grows faster than $ N^2 $. We shall need this property in the following form.
\begin{lemma}
    \label{lem:locality}
    Let a pair of compact sets $ A^{(1)} $, $ A^{(2)} \subset \mathbb R^p $ be metrically separated; let further $ \{\omega_N\subset A\colon N\in \mathfrak N\} $  be a sequence for which the limits  
    \[
        \lim_{\mathfrak N \ni N \to \infty} \frac{\#(\omega_N\cap A^{(i)})}{N} = \beta^{(i)}, \quad i =1,2.
    \]
    exist. Then 
    \[
        \begin{aligned}
            \liminf_{\mathfrak N \ni N \to \infty} \frac{E_s(\omega_N)}{N^{1+s/d}} &\geq\\
            \left(\beta^{(1)}\right)^{1+s/d}& \liminf_{\mathfrak N\ni  N\to\infty}  \frac{E_s(\omega_N \cap A^{(1)})}{\#(\omega_N\cap A^{(1)})^{1+s/d}}
            + \left(\beta^{(2)}\right)^{1+s/d} \liminf_{\mathfrak N\ni N\to\infty} \frac{E_s(\omega_N \cap A^{(2)})}{\#(\omega_N\cap A^{(2)})^{1+s/d}}.
        \end{aligned}
    \]
\end{lemma}
\begin{proof}
    We observe that with $ \sigma= \dist( A^{(1)}, A^{(2)} ) $,
    \[
    \left|{E_s(\omega_N)} -  \left({E_s(\omega_N \cap A^{(1)})} +  {E_s(\omega_N \cap A^{(2)})}\right) \right| =
    \sum_{\substack{\bs x_i \in A^{(1)}\\, \bs x_j\in A^{(2)}}}  |\bs x_i - \bs x_j|^{-s}
    \leq \sigma^{-s} N^2,
    \]
    and use the definition of $ \beta^{(i)}, \, i=1,2$, to obtain the desired equality.
\end{proof}

This is particularly useful for self-similar fractals satisfying the open set property. Consider such a  fractal $ A $; since $ \psi_m(V), \, 1\leq m \leq M, $ are pairwise disjoint for an open set $ V $ containing $ A $, there exists a $ \sigma >0 $ such that $ \dist(\psi_i(A),\psi_j(A)) \geq\sigma $ for $ i\neq j $. Following \cite{MR867284}, we will write 
\[
    A_{m_1\ldots m_l} := \psi_{m_1}\circ\ldots\circ \psi_{m_l}(A), \qquad 1\leq m_i\leq M, \quad l\geq 1.
\]
Then $ \dist(A_{m_1\ldots m_l},\, A_{m'_1\ldots m'_l}) \geq r_{m_1}\ldots r_{m_k}\sigma $, where $ k= \min \{i: m_i\neq m_i' \} $, so for a fixed $ M $ in the expression
\begin{equation*}
    A = \bigcup_{m_1,\ldots, m_l=1}^M A_{m_1\ldots m_l}
\end{equation*}
not only the union is disjoint, but also the sets $A_{m_1\ldots m_l}$ are metrically separated. The following lemma is technical, and we give its proof for the convenience of the reader.
\begin{lemma}
    \label{lem:splitting}
    If $ \{\mu_N : N \in \mathfrak N\} $ is a sequence of probability measures on the set $ A $, which for every $ l\geq 1 $ satisfies
    \[
        \lim_{\mathfrak N \ni N \to \infty} \mu_N(A_{m_1\ldots m_l}) = \mu(A_{m_1\ldots m_l}), \qquad 1\leq m_1,\ldots, m_l \leq M,
    \]
    for another probability measure $ \mu $ on $ A $, then 
    \[
        \mu_N \weakto \mu, \qquad \mathfrak N \ni N \to \infty.
    \] 
\end{lemma}
\begin{proof}
    Fix an $ f\in C(A) $; since $ A $ is compact, $ f $ is uniformly continuous on $ A $. For a fixed $\epsilon>0$, there exists an $ L_0\in \mathbb N $ such that $ |f(x) - f(y)| < \epsilon $ whenever $ x,y\in A_{m_1,\ldots,m_l} $ for any $ l \geq L_0 $ and any set of indices $ 0\leq m_1,\ldots,m_l\leq M $; this is possible due to 
    \[
        \diam(A_{m_1,\ldots,m_l}) \leq r_{m_1}\ldots r_{m_l} \diam(A) \leq \left(\max_{1\leq m \leq M} r_m\right)^l \diam(A).
    \]
    Fix an $ l\geq L_0 $ until the end of this proof, then pick an $ N_0 \in \mathfrak N $ so that for every $ N\geq N_0 $, there holds
    \[
        |\mu_N(A_{m_1\ldots m_l}) - \mu(A_{m_1\ldots m_l}) | < \epsilon/M^l, \qquad 1\leq m_1,\ldots,m_l \leq M.
    \]
    Finally, let us write $ f_{m_1\ldots m_l} := \min_{A_{m_1\ldots m_l}} f(x) $ for brevity. Then for $ N\geq N_0 $,
    \[
        \begin{aligned}
            \left|\int_A f(x) d\mu_N(x)\right. &-\left.\int_A f(x) d\mu(x)\right|\\
                                               &\leq \sum_{m_1,\ldots,m_l=1}^M \left|\int_{A_m} (f(x) - f_{m_1\ldots m_l}) d\mu_N(x)\right. -\left.\int_{A_m} (f(x) - f_{m_1\ldots m_l}) d\mu(x) \right|\\
                                               & + \sum_{m_1,\ldots,m_l=1}^M \left|(\mu_N(A_m) - \mu(A_m))f_{m_1\ldots m_l}\right|\\
                                               &\leq 2\epsilon + \epsilon \|f\|_\infty,
        \end{aligned}
    \]
    where the estimate for the first sum used that both $ \mu_N $ and $ \mu $ are probability measures. This proves the desired statement.
\end{proof}
Note that the converse is also true: since the sets $ A_{m_1,\ldots,m_l} $ are metrically separated, convergence $ \mu_N \weakto \mu $ of measures supported on $ A $ immediately implies (by Urysohn's lemma) $ \mu_N(A_{m_1\ldots m_l}) \to \mu(A_{m_1\ldots m_l}) $ for all $ l\geq 1 $ and all indices $ 1\leq m_1,\ldots,m_l\leq M $. 

The proof of the following statement follows a well-known approach \cite[Theorem 2]{Hardin2012, MR1458327}, and can be considered standard.

\begin{proposition}
    \label{lem:lowerupper}
    If $ A $ is a compact $ d $-regular set, then $ 0< \g(A) \leq \G(A) < \infty $.
\end{proposition}
The above proposition can be somewhat strengthened, to obtain uniform upper and lower bounds on 
\[
    \frac{\e_s(\omega_N)}{N^{1+s/d}}, \qquad N\geq 2;
\]
furthermore, each bound requires only one of the inequalities in \eqref{eq:dregular}.  In addition, for any  sequence of configurations $ \omega_N, \, N \in \mathfrak N, $ with
\[
    \lim_{\mathfrak N\ni N \to \infty} \frac{E_s(\omega_N)}{N^{1+s/d}} < \infty, 
\]
every weak$ ^*  $ cluster point of $ \nu_N, \, N\in \mathfrak N $, must be  absolutely continuous with respect to $ \h_d $ on $ A $. Lastly, we will need the following standard estimate.
\begin{corollary}
    \label{cor:point_energy}
    Suppose $ A $ is a compact $ d $-regular set,  $ \omega_N = \{\bs x_i : 1\leq i \leq N \} \subset A $, and $ s >d $. Then the minimal point energy of $ \omega_N $ is bounded by:
    \[
        \min_{\bs x \in A}\sum_{j=1}^N |\bs x - \bs x_j|^{-s} \leq CN^{s/d},
    \]
    where $ C  $ depends only on $ A,s,d $.
\end{corollary}
\begin{proof}[\textbf{Proof of \autoref{thm:weakstar}.}]
    In view of the weak$ ^* $ compactness of probability measures in $ A $, to establish existence of the weak$ ^* $ limit of $ \bbline\nu_N, \, N\in \n $, it suffices to show that any cluster point of $\bbline\nu_N, \ N\in \n, $ in the weak$ ^* $ topology is $h_d$ which is defined in \eqref{eq:defhd} (see \cite[Proposition A.2.7]{MR1070713}). To that end, consider a subsequence of $ \n $ for which the empirical measures $ \bbline \nu_N $ converge to a cluster point $ \mu $; for simplicity we shall use the same notation $ \n $ for this subsequence. 

    As discussed above,  $ \bbline \nu_N (A_{m_1\ldots m_l}) \to  \mu(A_{m_1\ldots m_l}),\, \n\ni N\to \infty $; this ensures that the quantities 
    \[
        \bbeta_m : = \mu(A_m) = \lim_{\n\ni N\to\infty} \bbline\nu_N(A_m) = \lim_{\n\ni N\to\infty} \frac{\#(\bbline\omega_N\cap A_m)}{N}, \qquad m=1,\ldots, M,
    \]
    are well-defined. From \eqref{eq:lower}, separation of $ \{A_m\} $, and Lemma \ref{lem:locality} follows
    \[
        \begin{aligned}
            \g(A) &= \sum_{m=1}^M \lim_{\n\ni N\to\infty} \frac{E_s(\bbline\omega_N \cap A_m)}{N^{1+s/d}} \geq \sum_{m=1}^M \bbeta_m^{1+s/d} \liminf_{\n\ni N\to\infty} \frac{E_s(\bbline\omega_N \cap A_m)}{\#(\bbline\omega_N\cap A_m)^{1+s/d}}\\
                  &\geq \sum_{m=1}^M \bbeta_m^{1+s/d} r_m^{-s} \g(A).
        \end{aligned}
    \]
    Consider the RHS in the last inequality. As a function of $ \{\bbeta_m\} $, it satisfies the constraint $ \sum_m \bbeta_m =1 $; note also that by the defining property \eqref{eq:dim} of $ d $, there holds $ \sum_m R_m =1 $  with $ R_m:= r_m^d,$  $ 1\leq m \leq M $. We have 
    \begin{equation}
        \label{eq:optimize}
        \begin{aligned}
            \g(A) &\geq \inf \left\{\sum_{m=1}^M \beta_m^{1+s/d} R_m^{-s/d}\colon  \sum_{m=1}^M \beta_m = 1 \right\} \g(A).  \\
        \end{aligned}
    \end{equation}
    Level sets of the function $ \sum_m \beta_m^{1+s/d} R_m^{-s/d} $ are convex, so the infimum is attained and unique; it is easy to check that the solution is at  $ \beta_m = R_m = r_m^d, \, 1\leq m \leq M, $ and the minimal value is $ 1 $. Indeed, the corresponding Lagrangian is
    \[
        L(\beta_1,\ldots,\beta_M, \lambda) := \sum_{m=1}^M \beta_m^{1+s/d} R_m^{-s/d} - \lambda\sum_{m=1}^M \beta_m,
    \]
    hence
    \[
        \nabla L_{\beta_m} = (1+s/d)\left(\frac{\beta_m}{R_m}\right)^{s/d} - \lambda, \quad 1\leq m \leq M,
    \]
    and it remains to use $ \beta_m \geq 0,\, 1\leq m \leq M $,  and $ \sum_m R_m =1 $, to conclude $ \beta_m = R_m, \, 1\leq m \leq M $.

    Since $ 0< \g(A) < \infty $ by Lemma \ref{lem:lowerupper}, from \eqref{eq:optimize} it follows
    \[
        \bbeta_m = r_m^d, \qquad m=1,\ldots, M.  
    \]
    Note that this argument shows also 
    \[
        \lim_{\n\ni N\to\infty} \frac{E_s(\bbline\omega_N \cap A_m)}{\left(\#(\bbline\omega_N\cap A_m)\right)^{1+s/d}} = \g(A),
    \]
    so the above can be repeated recursively for sets $ A_{m_1\ldots m_l} $. Namely, for every $ l\geq 1 $ and $ 1\leq m, m_1,\ldots, m_l \leq M $,
    \[
        \mu(A_{m\mkern1mu m_1\ldots m_l}) =: \bbeta_{m\mkern1mu m_1\ldots m_l} = r_m^d\mkern1mu\bbeta_{m_1\ldots m_l}.
    \]
    Observe further that $ h_d $ satisfies
    \[
        h_d(A_{m\mkern1mu m_1\ldots m_l}) = r_m^d h_d (A_{m_1\ldots m_l})
    \]
    by definition, so by Lemma \ref{lem:splitting} follows that every weak$ ^* $ cluster point of $ \bbline\nu_N,\, N\in\n,$ is $ h_d $, as desired. 
\end{proof} 
%TODO: extend to clopen sets
\begin{proof}[\textbf{Proof of Theorem~\ref{thm:subseq}.}]
    Note that setting equal contraction ratios $ r_1 = \ldots = r_m = r $ in \eqref{eq:dim} gives $ r^{-s} = M^{s/d} $.
    Consider the set function 
    \[
        \psi: \bs x \mapsto \bigcup_{m=1}^M \psi_m(\bs x), \quad \bs x \in  A,
    \]
    and denote
    \[
        \psi(\omega_N) := \bigcup_{\bs x\in \omega_N} \psi(x).
    \]
    It follows from the open set condition that the union above is metrically separated; as before, we denote the separation distance by $ \sigma $. Observe that the definition of a similitude implies $ \# (\psi(\omega_N)) = M\#(\omega_N) $. We then have for any configuration $ \omega_N,\, N \geq 2, $
    \[
        \begin{aligned}
            \e_s(A,MN) &\leq E_s(\psi(\omega_N)) \leq M r^{-s} E_s(\omega_N) + \sigma^{-s} N^2 M^2\\
                               & = M^{1+s/d} E_s(\omega_N) + \sigma^{-s} N^2 M^2,
    \end{aligned}
    \]
    and repeated application of the second inequality yields
    \[
        \begin{aligned}
            \e_s(A,M^kN) &\leq E_s[ \psi(\psi^{(k-1)}(\omega_N)) ] \leq M^{1+s/d} E_s(\psi^{(k-1)}(\omega_N)) + \sigma^{-s} (M^{k-1}N)^2 M^2 \\
                                 & \leq (M^2)^{1+s/d} E_s(\psi^{(k-2)}(\omega_N)) + M^{1+s/d}\sigma^{-s} (M^{k-2}N)^2 M^2 + \sigma^{-s} (M^{k-1}N)^2 M^2\\
                                 & \leq \ldots \\
                                 & \leq (M^k)^{1+s/d} E_s(\omega_N) + \sigma^{-s}N^2 \sum_{l=1}^k (M^{l-1})^{1+s/d} (M^{k-l})^2 M^2.
    \end{aligned}
    \]
    Estimating the geometric series in the last inequality, we obtain
    \begin{equation}
        \label{eq:MNbound} 
        \begin{aligned}
            \e_s(A,M^kN) & \leq (M^k)^{1+s/d} E_s(\omega_N) + \sigma^{-s}N^2 M^{2k+1 - s/d} \sum_{l=1}^k M^{l(s/d-1)}  \\
                         & \leq (M^k)^{1+s/d} E_s(\omega_N) + \sigma^{-s}N^2 M^{2k + 1 - s/d} \frac{M^{(k+1)(s/d-1)}-1}{M^{s/d-1} -1}\\
                         & \leq (M^k)^{1+s/d} E_s(\omega_N) + \frac {N^{1-s/d}}{\sigma^{s}(M^{s/d-1}-1)}  \left(M^{k} N\right)^{(1+s/d)}.
        \end{aligned}
    \end{equation}
    Let now $ \epsilon >0 $ fixed; find $ \omega_{N_0} $ such that $ N_0 \in \mathfrak M $ and 
    \[
        \frac{E_s(\omega_{N_0})}{N_0^{1+s/d}} \leq \liminf_{\mathfrak M \ni N \to \infty} \frac{\e_s(A,N)}{N^{1+s/d}} + \epsilon,
    \]
    and in addition,  $ N_0^{1-s/d} < \epsilon\sigma^s (M^{s/d-1}-1) $.
    Then by \eqref{eq:MNbound} we have
    \[
        \frac{\e_s(A, N)}{N^{1+s/d}} \leq  \frac{E_s(\omega_{N_0})}{N_0^{1+s/d}} + \epsilon \leq \liminf_{\mathfrak M \ni N \to \infty} \frac{\e_s(A,N)}{N^{1+s/d}} + 2\epsilon, \qquad \mathfrak M\ni N \geq N_0.
    \]
    This proves the desired statement.
\end{proof} 

In the following lemma we write $ \mathfrak N(k), k \in \mathbb N, $ to denote the $ k $-th element of the sequence $ \mathfrak N \subset \mathbb N $; we say that $ \mathfrak N $ is \textit{majorized} by a sequence $ \mathfrak M $, if the inequality $ \mathfrak N(k) < \mathfrak M(k) $ holds for every $ k\geq 1 $.

\begin{lemma}
    If  $ \mathfrak M \subset \mathbb N $ is a sequence such that the limit
    \[
        \lim_{\mathfrak M \ni N \to \infty} \frac{\e_s(A, N)}{N^{1+s/d}}
    \]
    exists, then for any sequence of integers $ \mathfrak N \subset \mathbb Z $ with $ |\mathfrak N(k)| $ majorized by $ \mathfrak M $ and satisfying $ |\mathfrak N(k)| = o(\mathfrak M(k)),\, k\to \infty, $ there holds
    \begin{equation}
        \label{eq:perturbation}
        \lim_{(\mathfrak M + \mathfrak N) \ni N \to \infty} \frac{\e_s(A, N)}{N^{1+s/d}} 
        = \lim_{\mathfrak M \ni N \to \infty} \frac{\e_s(A, N)}{N^{1+s/d}}, 
    \end{equation}
    where the addition $ \mathfrak M + \mathfrak N $ is performed elementwise.
\end{lemma}
\begin{proof}[\textbf{Proof.}] First, observe that by passing to subsequences of $ \mathfrak M $ and $ \mathfrak N $, it suffices to assume $ \mathfrak N(k) \geq 0 $ and to show\eqref{eq:perturbation} for $ \mathfrak M + \mathfrak N $ and $ \mathfrak M - \mathfrak N $.
    If $\mathfrak N (k) \geq 0$, we have by the definition of $\e_s$, 
    \[
        \e_s [A, (\mathfrak M + \mathfrak N)(k)] \geq \e_s (A, \mathfrak M (k))
    \]
     Thus
    \begin{equation} 
        \label{eq:tag1-1}
                \begin{aligned}
            &\liminf_{(\mathfrak M + \mathfrak N) \ni N \to \infty} \frac{\e_s(A, N)}{N^{1+s/d}} \geq \lim_{k\to \infty} \frac{\e_s(A, \mathfrak M(k))}{(\mathfrak M(k) + \mathfrak N(k) )^{1+s/d}}\\
            & = \lim_{k\to \infty} \frac{\e_s(A, \mathfrak M(k))}{(\mathfrak M(k) )^{1+s/d}} \left(\frac{\mathfrak M(k) }{\mathfrak M(k) + \mathfrak N(k)} \right)^{1+s/d} = \lim_{\mathfrak M \ni N \to \infty} \frac{\e_s(A, N)}{N^{1+s/d}},
    \end{aligned}
   \end{equation}
    in view of $ \mathfrak N(k) = o(\mathfrak M(k))$. Similarly, 
    \begin{equation}\label{eq:tag2-1} 
                \limsup_{(\mathfrak M - \mathfrak N) \ni N \to \infty} \frac{\e_s(A, N)}{N^{1+s/d}}  \leq \lim_{\mathfrak M \ni N \to \infty} \frac{\e_s(A, N)}{N^{1+s/d}}.
    \end{equation}
    For the converse estimates, use Corollary~\ref{cor:point_energy} to conclude that for every $ N\in \mathbb N $ there holds
    \[
        \e_s( A, N+1 )  \leq \e_s(A, N) + C N^{s/d}.
    \]
    Applying this inequality $ \mathfrak N(k) $ times to $ \mathfrak M(k) $, we obtain
    \[
        \e_s[ A, (\mathfrak M + \mathfrak N)(k) ] \leq \e_s(A, \mathfrak M(k)) + \mathfrak N(k)C[\mathfrak M(k) + \mathfrak N(k)]^{s/d},
    \]
    which yields
		\begin{equation}\label{eq:tag1-2}
        \limsup_{(\mathfrak M + \mathfrak N) \ni N \to \infty} \frac{\e_s(A, N)}{N^{1+s/d}}  \leq \lim_{\mathfrak M \ni N \to \infty} \frac{\e_s(A, N)}{N^{1+s/d}}.
    \end{equation}
    Finally, applying Corollary~\ref{cor:point_energy} $ \mathfrak N(k) $ times to $ \mathfrak M(k) - \mathfrak N(k)$ gives
    \[
        \e_s[ A, \mathfrak M (k) ] \leq \e_s[A, (\mathfrak M - \mathfrak N)(k)] + \mathfrak N(k)C\mathfrak M(k)^{s/d},
    \]
    whence, using that $ \mathfrak N (k) = o(\mathfrak M(k)),\, k\to \infty $,
    \begin{equation}\label{eq:tag2-2}
        \liminf_{(\mathfrak M - \mathfrak N) \ni N \to \infty} \frac{\e_s(A, N)}{N^{1+s/d}}  \geq \lim_{\mathfrak M \ni N \to \infty} \frac{\e_s(A, N)}{N^{1+s/d}}.
    \end{equation}
    Combining \eqref{eq:tag1-1} with \eqref{eq:tag1-2} and \eqref{eq:tag2-1} with \eqref{eq:tag2-2}, we get the desired result.
\end{proof}
The proof of the previous lemma  implies the following.
\begin{corollary}
    \label{cor:perturbation}
    If $ \mathfrak M, \mathfrak N \subset \mathbb N $ are a pair of sequences such that 
    \[
        \mathfrak N(k) \leq \theta\,\mathfrak M(k), \qquad k\geq 1,
    \]
    then 
    \[
        \liminf_{(\mathfrak M + \mathfrak N) \ni N \to \infty} \frac{\e_s(A, N)}{N^{1+s/d}}  \geq  \liminf_{\mathfrak M \ni N \to \infty} \frac{\e_s(A, N)}{N^{1+s/d}} \cdot \left( \frac1{1+\theta} \right)^{1+s/d}
    \]
    and
    \[ 
        \limsup_{(\mathfrak M + \mathfrak N) \ni N \to \infty} \frac{\e_s(A, N)}{N^{1+s/d}}  \leq \limsup_{\mathfrak M \ni N \to \infty} \frac{\e_s(A, N)}{N^{1+s/d}} + \frac{C\theta}{1+\theta},
    \]
    where $ C $ is the same as in Corollary~\ref{cor:point_energy}.
\end{corollary}
\begin{proof}[\textbf{Proof of Theorem~\ref{thm:well_posed}.}]
    To show that $ \fun(\cdot) $ is well-defined, it is necessary to verify that (i) existence of the limit $ \fl{\mathfrak N} $ implies that of the limit $ \el{\mathfrak N} $, and (ii) the value of $ \el{\mathfrak N} $ is uniquely defined by $ \fl{\mathfrak N} $. To this end, fix a pair of sequences  $ \mathfrak N_1,\, \mathfrak N_2 \subset \mathbb N $ such that $\fl{\mathfrak N_1} = \fl{\mathfrak N_2}$.

    First assume that $ \mathfrak N_1,\, \mathfrak N_2 $ are multiples of (a subset of) the geometric series, that is, $ \mathfrak N_i = \{M^k n_i \colon  k \in \mathfrak K_i \}, \, i=1,2 $. Observe that \eqref{eq:logsequences} implies $ \{\log_M n_1\} = \{\log_M n_2\} $ and let for definiteness $ n_2 \geq n_1 $; then $ n_2 = M^{k_0} n_1 $ for some integer $ k_0 \geq 1 $. It follows that $ \mathfrak N_i \subset \mathfrak N_0,\, i=1,2, $  with $ \mathfrak N_0 = \{ M^k n_0 \colon  k \geq 1 \} $.
    By Theorem~\ref{thm:subseq}, the limit 
    \[ 
        \lim_{\mathfrak N_0 \ni N \to \infty} \frac{\e_s(A, N)}{N^{1+s/d}}
    \]
    exists, so it must be that the limits over subsequences of $ \mathfrak N_0 $
    \[ 
        \lim_{\mathfrak N_i \ni N \to \infty} \frac{\e_s(A, N)}{N^{1+s/d}},\qquad i=1,2,
    \]
    also exist and are equal, so the function $ \fun(\cdot) $ is well-defined on the subset of $ [0,1] $ of all the sequences  $ \mathfrak N $ with $ \mathfrak N = \{M^k n \colon  k \in \mathfrak K \}$.

    Now let $ \mathfrak N_1,\, \mathfrak N_2 \subset \mathbb N $ be arbitrary. Denote the common value of the limit $ a:= \fl{\mathfrak N_i}, \, i=1,2 $. We shall assume for definiteness that $ a\in [0,1) $; the case of $ a=1 $ can be handled similarly.
    %TODO: endpoints of [0,1]
    In order to bound $ \mathfrak N_i $ between two sequences of the type $ \{M^k n_i \colon  k \in \mathfrak K_i \} $, discussed above, 
    fix an $ \epsilon > 0 $ such that $ a + 2\epsilon < 1 $, and find an $ N_0 \in \mathbb N $, for which 
    \begin{equation}
        \label{eq:converged_log}
        \left|\{\log_M N_i\} - a \right| < \epsilon, \qquad   N_0 \leq N_i\in \mathfrak N_i, \quad i=1,2.  
    \end{equation}
    By the choice of $ \epsilon, $ the above equation gives $ \lfloor\{\log_M N_1\}\rfloor = \lfloor\{\log_M N_2\}\rfloor $  when $ N_0 \leq N_i\in \mathfrak N_i $.
    Now let $ n_i,\, i=1,2 $ be such that 
    \begin{equation}
        \label{eq:appr_intseq}
        \begin{aligned}
            a - 2\epsilon &\leq  \{\log_M n_1\} \leq a - \epsilon\\
            a + \epsilon &\leq  \{\log_M n_2\}  \leq a + 2\epsilon.
        \end{aligned}
    \end{equation}
    Replacing one of $ n_i,\, i=1,2, $ with its multiple, if necessary, we can guarantee that $ 0< \log_M n_2 - \log_M n_1 < 4\epsilon  $.
    Consider a pair of sequences $ \nn_i = \{M^k n_i \colon  k \geq \lceil \log_M N_0 \rceil  \}, \, i=1,2 $; observe that by the above argument, limits
    \[ 
        \el {\nn_i} =: L_i,\qquad i=1,2,
    \]
    along $ \nn_i,\, i=1,2, $ both exist, and the inequality 
    \[
        \nn_1( k) \leq N_i \leq \nn_2(k), \qquad  k=\lfloor \log_M N_i \rfloor, \quad  N_0  \leq N_i \in \mathfrak N_i,\, i=1,2,
    \]
    holds.  By the definition of $ \e_s $, and due to \eqref{eq:converged_log}--\eqref{eq:appr_intseq},
    \[ 
        \limsup_{\mathfrak N_i \ni N \to \infty} \frac{\e_s(A, N)}{N^{1+s/d}} 
        \leq \lim_{k \to \infty} \frac{\e_s(A, M^k n_2)}{(M^k n_1)^{1+s/d}} = \left(\frac{n_2}{n_1}\right)^{1+s/d} L_2,\qquad i=1,2,
    \]
    and 
    \[ 
        \liminf_{\mathfrak N_i \ni N \to \infty} \frac{\e_s(A, N)}{N^{1+s/d}} \geq \lim_{k \to \infty} \frac{\e_s(A, M^k n_1)}{(M^k n_2)^{1+s/d}} = \left(\frac{n_1}{n_2}\right)^{1+s/d} L_1,
        \qquad i=1,2.
    \]
    Combining the last two inequalities gives
    \[
        \left(\frac{n_1}{n_2}\right)^{1+s/d} L_1 \leq \liminf_{\mathfrak N_i \ni N \to \infty} \frac{\e_s(A, N)}{N^{1+s/d}}
        \leq  \limsup_{\mathfrak N_i \ni N \to \infty} \frac{\e_s(A, N)}{N^{1+s/d}} 
        \leq \left(\frac{n_2}{n_1}\right)^{1+s/d} L_2,
    \]
    so it suffices to show that $ L_2 $ can be made arbitrarily close to $ L_1 $ by taking $ \epsilon \to 0 $. The latter follows from Corollary~\ref{cor:perturbation}, and the choice of $ n_i,\, i=1,2 $: 
    \[ 
        0 \leq \frac{\nn_2(k) - \nn_1(k)}{\nn_1(k)} = \frac{n_2}{n_1} - 1 \leq M^{4\epsilon} -1.
    \]
    Taking $ \epsilon \to 0 $ shows both that $ \el{\mathfrak N_1} = \el{\mathfrak N_2} $, and that these two limits exist. The function $ \fun: [0,1] \to (0,\infty) $ is therefore well-defined. Note that repeating the above argument for  $ |\fl{\mathfrak N_1} - \fl{\mathfrak N_2}| < \epsilon  $  for a fixed positive $ \epsilon $ gives a bound on $ |\el{\mathfrak N_1} - \el{\mathfrak N_2}| $, which implies that $ \fun $ is continuous. This completes the proof. 
\end{proof}

\begin{proof}[\textbf{Proof of Theorem~\ref{thm:prec_lowupfract}.}]
    Assume without loss of generality that the diameter of the set $ A $ satisfies
    \[
        \diam(A) = 1.
    \]
    Denote the minimal value of the Riesz $ s $-energy on $ M $ points on $ A $ by $ \e_{s,M} := \mathcal E_s(A, M) $; recall also that $ \sigma $ is the lower bound on the distance between $ A_i,\, A_j $ when $ i\neq j $.
    With this assumption, the last inequality in \eqref{eq:MNbound} with $ N=M $ gives 
    \begin{equation}
        \label{eq:MboundL} 
        \begin{aligned}
            \e_s(A,M^{k+1})
            &  \leq M^{k(1+s/d)} \e_{s,M} + \sigma^{-s}\frac {M^{1-s/d}}{(M^{s/d-1}-1)}  M^{(k+1)(1+s/d)} \\ 
            & \leq  M^{k(1+s/d)}\sigma^{-s} M^2 + \sigma^{-s}\frac {M^{2}}{(M^{s/d-1}-1)}  M^{k(1+s/d)}\\
            & =  M^{k(1+s/d)}\sigma^{-s} M^2 \left(1 + \frac1{M^{s/d-1} -1}\right)\\
            & =  M^{(k+1)(1+s/d)}\frac{\sigma^{-s}}{M^{s/d-1} -1}.  \\
    \end{aligned}
    \end{equation}
    On the other hand, consider a configuration $ \omega_{M^{k+1}+M^k} $. The set $ A $ is partitioned by the $ M^{k+1} $ subsets 
    \[
        A_{m_1\ldots m_{k+1}}, \qquad 1\leq m_1,\ldots, m_{k+1} \leq M,
    \]
    %TODO: state better
    so by the pigeonhole principle, for at least $ M^k $ pairs $ i\neq j $, the points $ \bs x_i,\, \bs x_j \in \omega_{M^{k+1}+M^k} $ belong to the same subset $ A_{m_1\ldots m_{k+1}} $. Writing $ r $ for the common contraction ratio of the defining similitudes $ \{ \psi_m\colon  1\leq m\leq M \} $ preserving the set $ A $, we have
    \[
        \diam(A_{m_1\ldots m_{k+1}}) = r^{k+1}\diam(A) = r^{k+1}.
    \]
    Configuration $ \omega_{M^{k+1}+M^k} $ was chosen arbitrarily, so it follows,
    \begin{equation}
        \label{eq:MboundU}
        \e_s(A,M^{k+1}+M^{k}) \geq M^k (r^{k+1})^{-s} = M^k (M^{s/d})^{k+1} = M^{s/d} (M^k)^{1+s/d},
    \end{equation}
    where we used that $ r^{-s} = M^{s/d} $ when all the contraction ratios are equal. Combining equations \eqref{eq:MboundL}--\eqref{eq:MboundU} gives
    \[
        \begin{aligned}
            \g(A) \big/ \G(A) 
            & \leq \limsup_{k\to \infty}\frac{\e_s(A,M^{k+1})}{ (M^{k+1})^{1+s/d} } 
            \,\big/\, \liminf_{k\to \infty}\frac{\e_s(A,M^{k+1}+M^{k})}{ (M^{k+1}+M^{k})^{1+s/d} }\\
            % & \leq  
            % \frac{\sigma^{-s} }{M^{s/d-1} -1} \big/
            % \liminf_{k\to \infty} \frac{M^{s/d} (M^k)^{1+s/d}}{(M^{k+1}+M^{k})^{1+s/d}}\\
            & = \frac{\sigma^{-s} }{M^{s/d-1} -1}  \big/ \frac{1}{M(1 + 1/M)^{1+s/d}}\\ 
            & = \frac{M(1 + 1/M)^{1+s/d} }{\sigma^{s}(M^{s/d-1} -1)}.  \\ 
        \end{aligned}
    \] 
    After substituting  $ 1/M = r^d $, the last inequality can be rewritten as
    \[ 
        \begin{aligned}
            \g(A) \big/ \G(A)
            & \leq \left(\frac{r(1 + r^d)^{1/d}}{\sigma}\right)^s \cdot \frac {M^{s/d-1}}{M^{s/d-1}-1}\cdot M(M+1)\\ 
            & = R^s \cdot \frac {M^{s/d-1}}{M^{s/d-1}-1}\cdot M(M+1).
    \end{aligned}
    \]
    Note that the second factor in the above equation is less than $ 2 $ when $ s > 2d $ holds (since $ M\geq 2 $); for an $ R <1 $, choosing the Riesz exponent as in \eqref{eq:sbound} makes the RHS less than $ 1 $, as desired.

\end{proof}
\begin{proof}[\textbf{Proof of Corollary~\ref{cor:ternary_cantor}.}]
    The proof repeats that of Theorem~\ref{thm:prec_lowupfract}, except for the simplified expression for  $ \mathcal E_{s,M} = \mathcal E_{s,2} = 1 $. Equations \eqref{eq:MboundL}--\eqref{eq:MboundU} become
    \[
        \begin{aligned}
            \e_s(A,2^{k+1}) & =  2^{(k+1)(1+s/d)}\frac{1}{2^2(2^{s/d-1} -1)},\\
            \e_s(A,2^{k+1}+2^{k}) & \geq 2^{s/d} (2^k)^{1+s/d},
        \end{aligned}
    \]
    respectively. Finally, from
    \[ 
        \begin{aligned}
            \g(A) \big/ \G(A)
            & \leq \frac{2(3/2)^{1+s/d}}{2^2(2^{s/d-1} -1)}  = \left(\frac{3}{4}\right)^{s/d} \cdot \frac{2^{s/d-1}}{2^{s/d-1} -1}\cdot\frac{3}{2}.
        \end{aligned}
    \]
    The RHS is a decreasing function of $ s $ and is less than $ 1 $ for $ s\geq 3d = 3\dim_H A = 3 \log_3 2 $, which completes the proof.
\end{proof}

\textbf{Acknowledgments.} The authors are grateful to Peter Grabner, Douglas Hardin, and Edward Saff for stimulating discussions. Part of this work was done while O.V.\ was in residence at the Institute for Computational and Experimental Research in Mathematics in Providence, RI during the ``Point Configurations in Geometry, Physics and Computer Science'' program, supported by the NSF grant DMS-1439786. Both authors would like to thank ICERM for the hospitality and for providing a welcoming environment for collaborative research.

%%%%%%%%%%%%%%%%%%%%%%%%%%%%%%%%%%%%%%%%%%%%%%%%%%%% 
%BibTex bibliography
\bibliographystyle{acm}
\bibliography{refs}
\end{document}